\documentclass[leqno,12pt]{amsart} % leqnoで数式番号を左側へ配置

\usepackage{amssymb} 
\usepackage{amsmath}
\usepackage{amsthm}
\usepackage{amscd}
\usepackage{bm}
\usepackage{graphicx,psfrag,wrapfig}
\theoremstyle{plain} % イタリック体
\newtheorem{theorem}{Theorem}[section] % 見出しはスモールキャップ
\newtheorem{lemma}[theorem]{Lemma}
\newtheorem{corollary}[theorem]{Corollary}
\newtheorem{proposition}[theorem]{Proposition}

\theoremstyle{definition} % ローマン体に変更

\newtheorem{remark}[theorem]{Remark}

\newcommand{\norm}[2]{\left|\!\left|#1\right|\!\right|_{#2}}

\newcommand{\Widim}{\mathrm{Widim}}
\newcommand{\Diam}{\mathrm{Diam}}

\begin{document}

\title[Macroscopic dimension of the $\ell^p$-ball]
{Macroscopic dimension of the $\ell^p$-ball with respect to the $\ell^q$-norm} % 論文タイトル []内は柱用

\author[Masaki Tsukamoto]{Masaki Tsukamoto$^*$} % 第一著者名 []内は柱用

%%%%%%%%%%%%%%%%%%% 脚注 %%%%%%%%%%%%%%%%%%%%%%%%%%%%%%
\subjclass[2000]{46B20}
% \subjclass[2000]コマンドが効かない(amsart.clsのバージョンが古い)場合は
% 以下を利用すること
%\renewcommand{\thefootnote}{\fnsymbol{footnote}}
%\footnote[0]{2000\textit{ Mathematics Subjet Classification}.
%Primary 00; Secondary 00.}

\keywords{$\ell^p$-space, Widim, mean dimension}

\thanks{$^*$Supported by Grant-in-Aid for JSPS Fellows (19$\cdot$1530)
from Japan Society for the Promotion of Science.}

\date{\today}
%%%%%%%%%%%% 著者所属 %%%%%%%%%%%%%

%\address{ % 第一著者
%Mathematical Institute \endgraf % 適当な箇所で改行
%Tohoku University \endgraf
%Sendai 980-8578 \endgraf
%Japan
%}
%\email{}

%%%%%%%%%%%%%%%%%%%%%%%%%%%%%%%%%%%%%%%%%%%%%%%%%%%%%%%

\maketitle

\begin{abstract}
We show estimates of the ``macroscopic dimension" of the 
$\ell^p$-ball with respect to the $\ell^q$-norm.
\end{abstract}

\section{Introduction}
\subsection{Macroscopic dimension}
Let $(X,d)$ be a compact metric space, $Y$ a topological space.
For $\varepsilon >0$, a continuous map $f:X\to Y$ is called an $\varepsilon$-embedding if 
$\Diam f^{-1}(y) \leq \varepsilon$ for all $y\in Y$.
Following Gromov \cite[p. 321]{Gromov}, we define the ``width dimension" $\Widim_\varepsilon X$ as the minimum integer 
$n$ such that there exist an $n$-dimensional polyhedron $P$ and an $\varepsilon$-embedding 
$f:X\to P$.
When we need to make the used distance $d$ explicit, we use the notation $\Widim_\varepsilon(X,d)$.
If we let $\varepsilon \to 0$, then $\Widim_\varepsilon$ gives the usual covering dimension:
\[ \lim_{\varepsilon\to 0} \Widim_\varepsilon X = \dim X .\]
$\Widim_\varepsilon X$ is a ``macroscopic" dimension of $X$ at the scale $\geq \varepsilon$
(cf. Gromov \cite[p. 341]{Gromov}).
It discards the information of $X$ ``smaller than $\varepsilon$".
For example, $[0,1]\times [0,\varepsilon]$ (with the Euclidean distance)
 macroscopically looks one-dimensional ($\varepsilon <1$):
\[ \Widim_\varepsilon [0,1]\times [0,\varepsilon] = 1 .\]

Using this notion, Gromov \cite{Gromov} defines ``mean dimension".
And he proposed open problems about this $\Widim_\varepsilon$ (see \cite[pp. 333-334]{Gromov}).
In this paper we give (partial) answers to some of them.

In \cite[p. 333]{Gromov}, he asks whether the simplex 
$\Delta^{n-1} :=\{x\in \mathbb{R}^n|\, x_k \geq 0 \,(1\leq k\leq n),\, \sum_{k=1}^{n} x_k = 1\}$ satisfies
\begin{equation}\label{eq: widim for simplex}
 \Widim_\varepsilon \Delta^{n-1} \sim \mathrm{const}_\varepsilon\, n .
\end{equation}
Our main result below gives the answer: 
If we consider the standard Euclidean distance on $\Delta^{n-1}$, 
then (\ref{eq: widim for simplex}) does \textit{not} hold.

In \cite[p. 333]{Gromov}, he also asks what is $\Widim_\varepsilon B_{\ell^p} (\mathbb{R}^n)$ 
with respect to the $\ell^q$-norm, where (for $1\leq p$)
\[ B_{\ell^p} (\mathbb{R}^n) := \left\{ x\in \mathbb{R}^n|\, \sum_{k=1}^n |x_k|^p \leq 1 \right\} .\]
Our main result concerns this problem.
For $1\leq q\leq \infty$, let $d_{\ell^q}$ be the $\ell^q$-distance on $\mathbb{R}^n$ given by 
\[ d_{\ell^q}(x, y) := \left( \sum_{k=1}^n |x_k - y_k|^q \right)^{1/q} .\]
We want to know the value of $\Widim_\varepsilon (B_{\ell^p}(\mathbb{R}^n), d_{\ell^q})$.
Especially we are interested in the behavior of $\Widim_\varepsilon (B_{\ell^p}(\mathbb{R}^n), d_{\ell^q})$
as $n \to \infty$ for small (but fixed) $\varepsilon$.
When $q=p$, we have (from ``Widim inequality" in \cite[p. 333]{Gromov}) 
\begin{equation}\label{eq: widim inequality}
 \Widim_\varepsilon (B_{\ell^p}(\mathbb{R}^n), d_{\ell^p}) = n \quad \text{for all $\varepsilon < 1$} .
\end{equation} 
(For its proof, see Gromov \cite[p. 333]{Gromov}, Gournay \cite[Lemma 2.5]{Gournay} or Tsukamoto \cite[Appendix A]{T3}.)
More generally, if $1\leq q \leq p \leq \infty$, then $d_{\ell^p} \leq d_{\ell^q}$ and hence
\begin{equation}\label{widim estimate for 1 leq q leq p}
 \Widim_\varepsilon (B_{\ell^p}(\mathbb{R}^n), d_{\ell^q}) = n \quad \text{for all $\varepsilon < 1$} .
\end{equation}
I think this is a satisfactory answer.
(For the case of $\varepsilon \geq 1$, there are still problems; see Gournay \cite{Gournay}.)
So the problem is the case of $1 \leq p < q \leq \infty$.
Our main result is the following:
\begin{theorem}\label{main theorem}
Let $1\leq p < q \leq \infty$ ($q$ may be $\infty$). 
We define $r\, (\geq p)$ by $1/p - 1/q = 1/r$.
For any $\varepsilon >0$ and $n \geq 1$, we have 
\begin{equation}\label{eq: main inequality}
 \Widim_\varepsilon (B_{\ell^p}(\mathbb{R}^n), d_{\ell^q}) 
\leq \min (n, \lceil (2/\varepsilon)^r \rceil -1 ) ,
\end{equation}
where $\lceil (2/\varepsilon)^r \rceil$ denotes the smallest integer $\geq (2/\varepsilon)^r$.
Note that the right-hand-side of (\ref{eq: main inequality}) 
becomes constant for large $n$ (and fixed $\varepsilon$).
Therefore $\Widim_\varepsilon (B_{\ell^p}(\mathbb{R}^n), d_{\ell^q})$ becomes stable as $n\to \infty$.
\end{theorem}
This result makes a sharp contrast with the above (\ref{widim estimate for 1 leq q leq p}).
For the simplex $\Delta^{n-1} \subset \mathbb{R}^n$ we have 
\[ \Widim_\varepsilon \Delta^{n-1} \leq \Widim_\varepsilon (B_{\ell^1}(\mathbb{R}^n), d_{\ell^2})
\leq \lceil (2/\varepsilon)^2 \rceil -1 .\]
Therefore (\ref{eq: widim for simplex}) does not hold.
Actually this result means that the ``macroscopic dimension" of $\Delta^{n-1}$  becomes constant 
for large $n$.

When $q =\infty$, we can prove that the inequality (\ref{eq: main inequality}) actually becomes an equality:
\begin{corollary}\label{cor: when q = infty}
For $1\leq p <\infty$,
\[ \Widim_\varepsilon (B_{\ell^p}(\mathbb{R}^n), d_{\ell^\infty}) 
= \min (n, \lceil (2/\varepsilon)^p \rceil -1 ) .\]
\end{corollary}
This result was already obtained by A. Gournay \cite[Proposition 1.3]{Gournay};
see Remark \ref{remark about the paper} at the end of the introduction.
For general $q >p$, I don't have an exact formula. But we can prove the following 
asymptotic result as a corollary of Theorem \ref{main theorem}.
\begin{corollary} \label{cor: asymptotics}
For $1\leq p < q \leq \infty$,
\[ \lim_{\varepsilon \to 0} 
\left( \lim_{n\to \infty} \log \Widim_{\varepsilon}(B_{\ell^p}(\mathbb{R}^n), d_{\ell^q}) / |\log \varepsilon | \right)
= r = \frac{pq}{q-p} .\]
Note that the limit $\lim_{n\to \infty} \log \Widim_{\varepsilon}(B_{\ell^p}(\mathbb{R}^n), d_{\ell^q})$ exists  
because $\Widim_{\varepsilon}(B_{\ell^p}(\mathbb{R}^n), d_{\ell^q})$ is monotone non-decreasing in $n$ and has an 
upper bound independent of $n$.
\end{corollary}
\begin{remark}
Gournay \cite[Example 3.1]{Gournay} shows
$\Widim_\varepsilon (B_{\ell^1}(\mathbb{R}^2), d_{\ell^p}) = 2$
for $\varepsilon < 2^{1/p}$.
\end{remark}

\subsection{Mean dimension theory}
Theorem \ref{main theorem} has an application to Gromov's mean dimension theory.
Let $\Gamma$ be a infinite countable group.
For $1\leq p \leq \infty$, let $\ell^p(\Gamma) \subset \mathbb{R}^\Gamma$ be the 
$\ell^p$-space, $B(\ell^p(\Gamma)) \subset \ell^p(\Gamma)$ the unit ball 
(in the $\ell^p$-norm).
We consider the natural right action of $\Gamma$ on $\ell^p(\Gamma)$ (and $B(\ell^p(\Gamma))$):
\[ (x\cdot\delta)_\gamma := x_{\delta\gamma} \quad 
\text{for $x= (x_\gamma)_{\gamma\in \Gamma}\in \ell^p(\Gamma)$ and $\delta\in\Gamma$} . \]
We give the standard product topology on $\mathbb{R}^\Gamma$, and consider the restriction of this 
topology to $B(\ell^p(\Gamma)) \subset \mathbb{R}^\Gamma$.
(This topology coincides with the restriction of weak topology of $\ell^p(\Gamma)$ 
for $p>1$.)
Then $B(\ell^p(\Gamma))$ becomes compact and metrizable.
(The $\Gamma$-action on $B(\ell^p(\Gamma))$ is continuous.)
Let $d$ be the distance on $B(\ell^p(\Gamma))$ compatible with the topology.
For a finite subset $\Omega \subset \Gamma$ we define a distance $d_\Omega$ 
on $B(\ell^p(\Gamma))$ by 
\[d_\Omega(x,y) := \max_{\gamma\in \Omega} d(x\gamma, y\gamma) .\]
We are interested in the growth behavior of 
$\Widim_\varepsilon (B(\ell^p(\Gamma)), d_\Omega)$ as $|\Omega| \to \infty$.
In particular, if $\Gamma$ is finitely generated and has an amenable sequence $\{\Omega_i\}_{i\geq 1}$
(in the sense of \cite[p. 335]{Gromov}), the mean dimension is defined by 
(see \cite[pp. 335-336]{Gromov})
\[ \dim(B(\ell^p(\Gamma)) :\Gamma) = \lim_{\varepsilon\to 0}
 \lim_{i\to \infty} \Widim_\varepsilon (B(\ell^p (\Gamma )), d_{\Omega_i})/|\Omega_i| .\]
As a corollary of Theorem \ref{main theorem}, we get the following:
\begin{corollary}\label{corollary: mean dimension}
For $1\leq p < \infty$ and any $\varepsilon >0$, 
there is a positive constant $C(d, p, \varepsilon) <\infty$ 
(independent of $\Omega$) such that 
\begin{equation}\label{eq: stability of widim}
 \Widim_\varepsilon(B(\ell^p(\Gamma)), d_\Omega) \leq C(d, p, \varepsilon) 
\quad \text{for all finite set $\Omega \subset \Gamma$} .
\end{equation}
Namely, $\Widim_\varepsilon (B(\ell^p(\Gamma)), d_\Omega)$ becomes stable for large $\Omega \subset \Gamma$.
In particular, for a finitely generated infinite amenable group $\Gamma$
\begin{equation}\label{eq: mean dimension = 0}
\dim(B(\ell^p(\Gamma)) :\Gamma) = 0.
\end{equation}
\end{corollary}
(\ref{eq: mean dimension = 0}) is the answer to the question of Gromov in \cite[p. 340]{Gromov}.
Actually the above (\ref{eq: stability of widim}) is much stronger than (\ref{eq: mean dimension = 0}).
\begin{remark}\label{remark about the paper}
This paper is a revised version of the preprint \cite{Tsukamoto}.
A referee of \cite{Tsukamoto} pointed out that the above (\ref{eq: mean dimension = 0}) can be derived from 
the theorem of Lindenstrauss-Weiss \cite[Theorem 4.2]{Lindenstrauss-Weiss}.
This theorem tells us that if the topological entropy is finite then the mean dimension becomes $0$.
We can see that the topological entropy of $B(\ell^p(\Gamma))$ (under the $\Gamma$-action)
is $0$. Hence the mean dimension also becomes $0$.
I am most grateful to the referee of \cite{Tsukamoto} for pointing out this argument.
The essential part of the proof of Theorem \ref{main theorem} (and Corollary \ref{cor: when q = infty} and
Corollary \ref{cor: asymptotics})
is the construction of the continuous map $f:\mathbb{R}^n \to \mathbb{R}^n$ described in Section 
\ref{sec: main construction}.
This construction was already given in the preprint \cite{Tsukamoto}.
When I was writing this revised version of \cite{Tsukamoto}, 
I found the paper of A. Gournay \cite{Gournay}. \cite{Gournay} proves Corollary \ref{cor: when q = infty}
(\cite[Proposition 1.3]{Gournay})
by using essentially the same continuous map as mentioned above.
I submitted the paper \cite{Tsukamoto} to a certain journal in June of 2007
before \cite{Gournay} appeared on the arXiv in November of 2007.
And \cite{Tsukamoto} is quoted as one of the references in \cite{Gournay}.
\end{remark}

%%%%%%%%%%%%%%%%%%%%%%%%%%%%%%%%%%%%%%%%%%%%%%%%%%%%%%%%%%%%%%%%%%%%%%%%%%%%%%%%%%%%%%%%%%%%%%%%%%%%%%%%
%%%%%%%%%%%%%%%%%%%%%%%%%%%%%%%%%%%%%%%%%%%%%%%%%%%%%%%%%%%%%%%%%%%%%%%%%%%%%%%%%%%%%%%%%%%%%%%%%%%%%%%%

\section{preliminaries}
\begin{lemma}\label{preliminary 1}
For $s\geq 1$ and $x, y, z\geq 0$, if $x\geq y$, then 
\[ x^s + (y+z)^s \leq (x+z)^s + y^s .\] 
\end{lemma}
\begin{proof}
Set $\varphi(t) := (t+z)^s - t^s$ ($t\geq 0$).
Then $\varphi'(t) = s\{ (t+z)^{s-1} - t^{s-1}\} \geq 0$.
Hence $\varphi(y)\leq \varphi(x)$, i.e., 
$(y+z)^s - y^s \leq (x+z)^s - x^s$.
\end{proof}
\begin{lemma}\label{key lemma}
Let $s\geq 1$ and $c, t\geq 0$.
If real numbers $x_1, \cdots, x_n$ $(n\geq 1)$ satisfies 
\[ x_1 + \cdots + x_n \leq c, \quad 0\leq x_i \leq t \> (1\leq i \leq n) ,\]
then 
\[ x_1^s + \cdots + x_n^s \leq c\cdot t^{s-1} .\]
\end{lemma}
\begin{proof}
First we suppose $nt\leq c$. Then $x_1^s + \cdots + x_n^s \leq n\cdot t^s \leq c\cdot t^{s-1}$.

Next we suppose $nt > c$. Let $m := \lfloor c/t \rfloor$ be the maximum integer satisfying $mt\leq c$.
We have $0\leq m < n$ and $c-mt < t$.
Using Lemma \ref{preliminary 1}, we have 
\[ x_1^s + \cdots + x_n^s \leq \underbrace{t^s + \cdots + t^s}_{m} + (c-mt)^s
 \leq m t^s + t^{s-1}(c-mt) \leq c\cdot t^{s-1} .\]
\end{proof}

%%%%%%%%%%%%%%%%%%%%%%%%%%%%%%%%%%%%%%%%%%%%%%%%%%%%%%%%%%%%%%%%%%%%%%%%%%%%%%%%%%%%%%%%%%%%%%%%%%%%
%%%%%%%%%%%%%%%%%%%%%%%%%%%%%%%%%%%%%%%%%%%%%%%%%%%%%%%%%%%%%%%%%%%%%%%%%%%%%%%%%%%%%%%%%%%%%%%%%%%%

\section{Proof of Theorem \ref{main theorem}} \label{sec: main construction}
Let $S_n$ be the $n$-th symmetric group.
We define the group $G$ by 
\[ G := \{\pm 1\}^{n} \rtimes S_n .\]
The multiplication in $G$ is given by
\[ ((\varepsilon_1, \cdots, \varepsilon_n ), \sigma ) \cdot
((\varepsilon'_1, \cdots, \varepsilon'_n ), \sigma' ) :=
((\varepsilon_1 \varepsilon'_{\sigma^{-1}(1)}, \cdots, \varepsilon_n \varepsilon'_{\sigma^{-1}(n)}),
\sigma \sigma') \]
where $\varepsilon_1, \cdots, \varepsilon_n, \varepsilon'_1 \cdots, \varepsilon'_n \in \{\pm1\}$ 
and $\sigma, \sigma' \in S_n$.
$G$ acts on $\mathbb{R}^n$ by 
\[ ((\varepsilon_1, \cdots, \varepsilon_n ), \sigma ) \cdot (x_1, \cdots, x_n) 
:= (\varepsilon_1 x_{\sigma^{-1}(1)}, \cdots, \varepsilon_n x_{\sigma^{-1}(n)}) \]
where $((\varepsilon_1, \cdots, \varepsilon_n ), \sigma ) \in G$ 
and $(x_1, \cdots, x_n) \in \mathbb{R}^n$.
The action of $G$ on $\mathbb{R}^n$ preserves 
the $\ell^p$-ball $B_{\ell^p}(\mathbb{R}^n)$ and the $\ell^q$-distance $d_{\ell^q}(\cdot, \cdot)$.

We define $\mathbb{R}^n_{\geq 0}$ and $\varLambda_n$ by 
\begin{equation*}
 \begin{split}
 \mathbb{R}^n_{\geq 0} &:= \{ (x_1, \cdots, x_n)\in \mathbb{R}^n |\, x_i \geq 0 \; (1\leq i \leq n) \}, \\
 \varLambda_n &:= \{(x_1, \cdots, x_n) \in \mathbb{R}^n|\, x_1 \geq x_2 \geq \cdots \geq x_n\geq 0 \}.
 \end{split}
\end{equation*} 
The following can be easily checked:
\begin{lemma}\label{lemma: fixed point of G}
For $\varepsilon \in \{\pm 1\}^n$ and $x\in \mathbb{R}^n_{\geq 0}$, 
if $\varepsilon x\in \mathbb{R}^n_{\geq 0}$, then $\varepsilon x = x$.
For $\sigma\in S_n$ and $x\in \varLambda_n$, if $\sigma x\in \varLambda_n$, then $\sigma x = x$.
For $g = (\varepsilon, \sigma) \in G$ and $x\in \varLambda_n$, if $g x \in \varLambda_n$, then
$gx = \varepsilon (\sigma x) = \sigma x = x$.
\end{lemma}

Let $m, n$ be positive integers such that $1 \leq m < n$.
We define the continuous map $f_0:\varLambda_n \to \varLambda_n$ by 
\[ f_0(x_1, \cdots, x_n) := 
(x_1 -x_{m+1}, x_2 - x_{m+1}, \cdots, x_m - x_{m+1}, \underbrace{0, 0, \cdots, 0}_{n-m}) .\]
The following is the key fact for our construction:
\begin{lemma}\label{lemma: key for the construction}
For $g\in G$ and $x\in \varLambda_n$, if $gx\in \varLambda_n$ ($\Rightarrow gx =x$), then we have
\[ f_0(gx) = g f_0(x) .\]
\end{lemma}
\begin{proof}
First we consider the case of
$g = \varepsilon = (\varepsilon_1, \cdots, \varepsilon_n) \in \{\pm 1\}^n$.
If $x_{m+1} =0$, then
\[ f_0(\varepsilon x) = (\varepsilon_1 x_1, \cdots, \varepsilon_m x_m , 0, \cdots, 0) 
= \varepsilon f_0(x).\]
If $x_{m+1} >0$, then $\varepsilon_i = 1\; (1\leq i\leq m+1)$ because 
$\varepsilon_i x_i = x_i \geq x_{m+1} >0 \; (1\leq i\leq m+1)$.
Hence 
\[ f_0(\varepsilon x) = (x_1-x_{m+1}, \cdots, x_m - x_{m+1}, 0, \cdots, 0) = f_0(x) = 
\varepsilon f_0(x) .\]

Next we consider the case of $g = \sigma \in S_n$. $gx\in \varLambda_n$ implies 
$x_{\sigma(i)} = x_i \; (1\leq i\leq n)$. Set $y:= f_0(x)$.
Let $r\; (1\leq r\leq m+1)$ be the integer such that
\[ x_{r-1} > x_r = x_{r+1} = \cdots =x_{m+1} .\]
From $x_{\sigma(i)} = x_i \; (1\leq i\leq n)$, we have
\begin{equation*}
 \begin{split}
 &1\leq i < r\Rightarrow 1\leq \sigma(i) < r \Rightarrow y_{\sigma(i)} = x_{\sigma(i)} - x_{m+1} = y_i, \\
 &r\leq i \Rightarrow r\leq \sigma(i) \Rightarrow y_{\sigma(i)} = 0 = y_i.
 \end{split}
\end{equation*}
Hence we have $f_0(\sigma x) =  f_0(x) = \sigma f_0(x)$.

Finally we consider the case of $g =(\varepsilon, \sigma)\in G$.
Since $gx\in \varLambda_n$, we have $gx = \varepsilon (\sigma x) = \sigma x = x \in \varLambda_n$ 
(see Lemma \ref{lemma: fixed point of G}). Hence
\[ f_0(gx) = f_0(\varepsilon (\sigma x )) = \varepsilon f_0(\sigma x ) 
= \varepsilon \sigma f_0(x) = g f_0(x) .\]
\end{proof}
We define a continuous map $f:\mathbb{R}^n \to \mathbb{R}^n$ as follows;
For any $x\in \mathbb{R}^n$, there is a $g\in G$ such that $gx\in \varLambda_n$.
Then we define 
\[f(x):= g^{-1}f_0(gx).\]
From Lemma \ref{lemma: key for the construction}, this definition is well-defined. 
Since $\mathbb{R}^n = \bigcup_{g\in G} g\varLambda_n$ and
$f|_{g\varLambda_n} = g f_0 g^{-1} \; (g\in G)$ is continuous on $g\varLambda_n$, $f$ is continuous on $\mathbb{R}^n$.
Moreover $f$ is $G$-equivariant.

\begin{proposition}\label{proposition: key estimate}
Let $1 \leq p < q\leq \infty$.
For any $x\in B_{\ell^p}(\mathbb{R}^n)$, we have
\[ d_{\ell^q} (x, f(x)) \leq \left( \frac{1}{m+1}\right)^{\frac{1}{p} - \frac{1}{q}} .\]
Note that the right-hand side does not depend on $n$.
\end{proposition}
\begin{proof}
Since $f$ is $G$-equivariant and $d_{\ell^q}$ is $G$-invariant, 
we can suppose $x\in B_{\ell^p}(\mathbb{R}^n)\cap \varLambda_n$, i.e.
$x = (x_1, x_2, \cdots, x_n)$ with $x_1 \geq x_2\geq \cdots \geq x_n\geq 0$.
We have 
\[ f(x) = (x_1 - x_{m+1}, \cdots, x_m - x_{m+1}, 0, \cdots, 0) .\]
Hence 
\[ d_{\ell^q} (x, f(x)) = |\!|(\underbrace{x_{m+1}, \cdots, x_{m+1}}_{m+1},x_{m+2}, \cdots, x_n)|\!|_{\ell^q} .\]
Set $t:= x_{m+1}^p$ and $s:= q/p$. Since $x_1^p + \cdots + x_n^p \leq 1$ and 
$x_1 \geq x_2\geq \cdots \geq x_n\geq 0$, we have $t \leq 1/(m+1)$, $0\leq x_k^p \leq t$ ($m+1\leq k\leq n$)
and $x_{m+2}^p + \cdots + x_n^p \leq 1-(m+1)t$.
Using Lemma \ref{key lemma}, we have
\[ x_{m+2}^q + \cdots + x_n^q \leq \{1-(m+1)t\}t^{s-1} = t^{s-1} - (m+1)t^s .\]
Therefore
\[ d_{\ell^q} (x, f(x))^q = (m+1)x_{m+1}^q + x_{m+2}^q + \cdots + x_n^q \leq t^{s-1} \leq (1/(m+1))^{s-1}. \]
Thus 
\[ d_{\ell^q} (x, f(x)) \leq (1/(m+1))^{1/p -1/q} .\]
\end{proof}
\begin{proof}[Proof of Theorem \ref{main theorem}]
Set $m:= \min (n, \lceil (2/\varepsilon)^r \rceil -1 )$.
We will prove $\Widim_\varepsilon (B_{\ell^p}(\mathbb{R}^n), d_{\ell^q}) \leq m$.
If $n = m$, then the statement is trivial. 
Hence we suppose $n > m = \lceil (2/\varepsilon)^r \rceil -1$.
From $m+1 = \lceil (2/\varepsilon)^r \rceil \geq (2/\varepsilon)^r$ and $1/r = 1/p -1/q$,
\[ 2 \left( \frac{1}{m+1}\right)^{\frac{1}{p} - \frac{1}{q}} \leq \varepsilon .\]
We have
\[ f(\mathbb{R}^n) = \bigcup_{g\in G} g f(\varLambda_n) .\]
Note that $f(\varLambda_n) \subset \mathbb{R}^m := \{(x_1, \cdots, x_m, 0, \cdots, 0)\in \mathbb{R}^n\}$.
Proposition \ref{proposition: key estimate} implies that
\[f|_{B_{\ell^p}(\mathbb{R}^n)}: (B_{\ell^p}(\mathbb{R}^n), d_{\ell^q}) \to 
\bigcup_{g\in G}g\cdot\mathbb{R}^m
\text{ is a $2\left(\frac{1}{m+1}\right)^{\frac{1}{p} - \frac{1}{q}}$-embedding} .\]
Therefore we get $\Widim_\varepsilon (B_{\ell^p}(\mathbb{R}^n), d_{\ell^q}) \leq m$.
\end{proof}

%%%%%%%%%%%%%%%%%%%%%%%%%%%%%%%%%%%%%%%%%%%%%%%%%%%%%%%%%%%%%%%%%%%%%%%%%%%%%%%%%%%%%%%%%%%%%%%%%%%
%%%%%%%%%%%%%%%%%%%%%%%%%%%%%%%%%%%%%%%%%%%%%%%%%%%%%%%%%%%%%%%%%%%%%%%%%%%%%%%%%%%%%%%%%%%%%%%%%%%

\section{Proof of Corollaries \ref{cor: when q = infty} and \ref{cor: asymptotics}} 
\subsection{Proof of Corollary \ref{cor: when q = infty}}
We need the following result. (cf. Gromov \cite[p. 332]{Gromov}. 
For its proof, see Lindenstrauss-Weiss \cite[Lemma 3.2]{Lindenstrauss-Weiss}
 or Tsukamoto \cite[Example 4.1]{T2}.)
\begin{lemma}\label{lemma: fundamental widim estimate}
For $\varepsilon < 1$,
\[ \Widim_\varepsilon ([0,1]^n, d_{\ell^\infty}) = n  ,\]
where $d_{\ell^\infty}$ is the sup-distance given by 
$d_{\ell^\infty}(x, y) := \max_{i} |x_i-y_i|$.　
\end{lemma}
From this we get:
\begin{lemma}\label{lemma: widim of l^infty ball}
Let $B_{\ell^\infty} (\mathbb{R}^n, \rho)$ be the closed ball of radius $\rho$ 
centered at the origin in $\ell^\infty (\mathbb{R}^n)$ ($\rho >0$).
Then for $\varepsilon < 2\rho$ 
\[ \Widim_\varepsilon (B_{\ell^\infty}(\mathbb{R}^n, \rho), d_{\ell^\infty}) = n .\]
\end{lemma}
\begin{proof}
Consider the bijection 
\[ [0,1]^n \to B_{\ell^\infty}(\mathbb{R}^n, \rho) , \quad 
(x_1, \cdots , x_n) \mapsto ( 2\rho x_1 - \rho , \cdots, 2\rho x_n -\rho ) .\]
Then the statement easily follows from Lemma \ref{lemma: fundamental widim estimate}.
\end{proof}
\begin{proof}[Proof of Corollary \ref{cor: when q = infty}]
Set $m:= \min (n, \lceil (2/\varepsilon)^p \rceil -1 )$.
We already know (Theorem \ref{main theorem}) 
$\Widim_\varepsilon (B_{\ell^p}(\mathbb{R}^n), d_{\ell^\infty}) \leq m$.
We want to show $\Widim_\varepsilon (B_{\ell^p}(\mathbb{R}^n), d_{\ell^\infty}) \geq m$.
Note that for any real number $x$ we have $\lceil x \rceil -1 < x$. 
Hence $m \leq \lceil (2/\varepsilon)^p \rceil -1  < (2/\varepsilon)^p$.
Therefore $m (\varepsilon/2)^p < 1$.
Then if we choose $\rho > \varepsilon/2$ sufficiently close to $\varepsilon/2$, then ($m \leq n$) 
\[ B_{\ell^\infty}(\mathbb{R}^m, \rho ) \subset B_{\ell^p}(\mathbb{R}^n) .\]
(If $\varepsilon \geq 2$, then $m=0$ and $B_{\ell^\infty}(\mathbb{R}^m, \rho)$ is $\{0\}$.)
From Lemma \ref{lemma: widim of l^infty ball}, 
\[ \Widim_\varepsilon (B_{\ell^p}(\mathbb{R}^n), d_{\ell^\infty})
\geq \Widim_\varepsilon (B_{\ell^\infty}(\mathbb{R}^m, \rho), d_{\ell^\infty})  = m .\]
Essentially the same argument is given in Gournay \cite[pp. 5-6]{Gournay}.
\end{proof}

\subsection{Proof of Corollary \ref{cor: asymptotics}}
The following lemma easily follows from (\ref{eq: widim inequality})
\begin{lemma} \label{lemma: widim inequality}
Let $B_{\ell^q}(\mathbb{R}^n, \rho)$ be the closed ball of radius $\rho$ centered at the origin 
in $\ell^q(\mathbb{R}^n)$ ($1\leq q \leq \infty$ and $\rho >0$).
For $\varepsilon < \rho$,
\[ \Widim_\varepsilon (B_{\ell^q}(\mathbb{R}^n, \rho), d_{\ell^q}) = n .\]
\end{lemma}
\begin{proposition} \label{prop: lower bound}
For $1\leq p < q \leq \infty$,
\[ \min (n, \lceil \varepsilon^{-r} \rceil -1 ) \leq \Widim_\varepsilon (B_{\ell^p}(\mathbb{R}^n), d_{\ell^q}) ,\]
where $r$ is defined by $1/r = 1/p - 1/q$.
\end{proposition}
\begin{proof}
We can suppose $q <\infty$.
Set $m := \min (n, \lceil \varepsilon^{-r} \rceil -1 )$.
From H\"{o}lder's inequality, 
\[ (|x_1|^p + \cdots + |x_m|^p )^{1/p} \leq m^{1/r} (|x_1|^q + \cdots + |x_m|^q)^{1/q} .\]
As in the proof of Corollary \ref{cor: when q = infty}, we have 
$m \leq \lceil \varepsilon^{-r} \rceil -1 < \varepsilon^{-r}$, i.e. $m^{1/r} \varepsilon < 1$.
Therefore if we choose $\rho >\varepsilon$ sufficiently close to $\varepsilon$, then ($m\leq n$)
\[ B_{\ell^q}(\mathbb{R}^m, \rho ) \subset B_{\ell^p}(\mathbb{R}^n) .\]
From Lemma \ref{lemma: widim inequality},
\[ \Widim_\varepsilon (B_{\ell^p}(\mathbb{R}^n), d_{\ell^q}) \geq 
\Widim_\varepsilon (B_{\ell^q}(\mathbb{R}^m, \rho ), d_{\ell^q}) = m .\]
\end{proof}
\begin{proof}[Proof of Corollary \ref{cor: asymptotics}]
From Theorem \ref{main theorem} and Proposition \ref{prop: lower bound}, we have 
\[ \lceil \varepsilon^{-r} \rceil -1 
\leq \lim_{n\to \infty} \Widim_\varepsilon (B_{\ell^p}(\mathbb{R}^n), d_{\ell^q}) \leq 
\lceil (2/\varepsilon)^{r} \rceil -1 .\]
From this estimate, we can easily get the conclusion.
\end{proof}

%%%%%%%%%%%%%%%%%%%%%%%%%%%%%%%%%%%%%%%%%%%%%%%%%%%%%%%%%%%%%%%%%%%%%%%%%%%%%%%%%%%%%%%%%%%%%%%%%%%%
%%%%%%%%%%%%%%%%%%%%%%%%%%%%%%%%%%%%%%%%%%%%%%%%%%%%%%%%%%%%%%%%%%%%%%%%%%%%%%%%%%%%%%%%%%%%%%%%%%%%

\section{Proof of Corollary \ref{corollary: mean dimension}}
Let $1\leq p < \infty$ and $\varepsilon >0$. Set $X:= B(\ell^p(\Gamma))$.
To begin with, we want to fix a distance on $X$ (compatible with the topology).
Since $X$ is compact, if we prove (\ref{eq: stability of widim}) for one fixed 
distance, then (\ref{eq: stability of widim}) becomes valid for any distance on $X$.
Let $w:\Gamma \to \mathbb{R}_{>0}$ be a positive function satisfying 
\begin{equation*}
 \sum_{\gamma\in \Gamma} w(\gamma) \leq 1 .
\end{equation*}
We define the distance $d(\cdot, \cdot)$ on $X$ by
\[ d(x, y) := \sum_{\gamma\in \Gamma} w(\gamma) |x_\gamma - y_\gamma| \quad
 \text{for $x= (x_\gamma)_{\gamma\in \Gamma}$ and $y= (y_\gamma)_{\gamma\in \Gamma}$ in $X$}. \]
As in Section 1, we define the distance $d_\Omega$ on $X$ for a finite subset $\Omega\subset \Gamma$ by
\[ d_\Omega(x, y) := \max_{\gamma\in \Omega} d(x\gamma, y\gamma) .\]

For each $\delta\in \Gamma$, there is a finite set $\Omega_\delta \subset \Gamma$ such that 
\[ \sum_{\gamma\in \Gamma\setminus \Omega_\delta} w(\delta^{-1}\gamma) \leq \varepsilon /4.\]
Set $\Omega' := \bigcup_{\delta\in \Omega}\Omega_\delta$. $\Omega'$ is a finite set satisfying
\[ \sum_{\gamma\in \Gamma\setminus \Omega'} w(\delta^{-1}\gamma) \leq \varepsilon /4 \quad
\text{for any $\delta\in \Omega$}. \]
Set $c: = \lceil (4/\varepsilon)^p \rceil -1$.
Let $\pi:X\to B_{\ell^p}(\mathbb{R}^{\Omega'}) = \{x\in \mathbb{R}^{\Omega'}|\,\norm{x}{p}\leq 1\}$ 
be the natural projection.
From Theorem \ref{main theorem}, there are a
polyhedron $K$ of dimension $\leq c$ and an $\varepsilon /2$-embedding 
$f:(B_{\ell^p}(\mathbb{R}^{\Omega'}), d_{\ell\infty})\to K$.
Then $F:= f \circ \pi :(X, d_\Omega)\to K$ becomes an $\varepsilon$-embedding; 
If $F(x) = F(y)$, then $d_{\ell^\infty}(\pi(x), \pi(y)) \leq \varepsilon/2$ and for each $\delta\in \Omega$
\begin{equation*}
 \begin{split}
 d(x\delta, y\delta) &= \sum_{\gamma\in \Omega'}w(\delta^{-1}\gamma)|x_\gamma-y_\gamma|
  + \sum_{\gamma\in \Gamma\setminus \Omega'} w(\delta^{-1}\gamma)|x_\gamma-y_\gamma|, \\
  &\leq \frac{\varepsilon}{2}\sum_{\gamma\in \Omega'}w(\delta^{-1}\gamma) 
   + 2\sum_{\gamma\in \Gamma\setminus \Omega'} w(\delta^{-1}\gamma) ,\\
  &\leq \varepsilon /2 + \varepsilon /2 = \varepsilon.
 \end{split}
\end{equation*}
Hence $d_\Omega(x, y)\leq \varepsilon$.
Therefore,
\[ \Widim_\varepsilon (X, d_\Omega) \leq c .\]
This shows (\ref{eq: stability of widim}).
If $\Gamma$ has an amenable sequence $\{\Omega_i\}_{i\geq 1}$, then $|\Omega_i|\to \infty$ and hence
\[ \lim_{i\to \infty}\Widim_\varepsilon (X, d_{\Omega_i})/|\Omega_i| = 0.\]
This shows (\ref{eq: mean dimension = 0}).

\vspace{10mm}

\address{ Masaki Tsukamoto \endgraf
Department of Mathematics, Faculty of Science \endgraf
Kyoto University \endgraf
Kyoto 606-8502 \endgraf
Japan
}

\textit{E-mail address}: \texttt{tukamoto@math.kyoto-u.ac.jp}

\end{document}